\documentclass[12pt, reqno]{amsart}
\usepackage{amsmath, amsthm, amscd, amsfonts, amssymb, graphicx, color}
\usepackage{bm}
\usepackage[all,cmtip]{xy}
\usepackage{multirow,bigdelim}

\newtheorem{theorem}{Theorem}[section]
\newtheorem{lemma}[theorem]{Lemma}
\newtheorem{proposition}[theorem]{Proposition}
\newtheorem{corollary}[theorem]{Corollary}
\theoremstyle{definition}

\numberwithin{equation}{section}

\newcommand{\vp}{\varphi}

\newcommand{\clb}{\mathcal{B}}

\newcommand{\clh}{\mathcal{H}}

\newcommand{\cll}{\mathcal{L}}

\newcommand{\clr}{\mathcal{R}}
\newcommand{\cls}{\mathcal{S}}

\newcommand{\clz}{\mathcal{Z}}

\newcommand{\D}{\mathbb{D}}

\newcommand{\T}{\mathbb{T}}
\newcommand{\Z}{\mathbb{Z}}

\newcommand{\raro}{\rightarrow}

\begin{document}


\setcounter{page}{1}


\title[Partially isometric Toeplitz operators]{Partially isometric Toeplitz operators on the polydisc}

\author[Deepak]{Deepak K. D.}
\address{Indian Statistical Institute, Statistics and Mathematics Unit, 8th Mile, Mysore Road, Bangalore, 560059,
India}
\email{dpk.dkd@gmail.com }

\author[Pradhan]{Deepak Pradhan}
\address{Indian Statistical Institute, Statistics and Mathematics Unit, 8th Mile, Mysore Road, Bangalore, 560059,
India}
\email{deepak12pradhan@gmail.com}

\author[Sarkar]{Jaydeb Sarkar}
\address{Indian Statistical Institute, Statistics and Mathematics Unit, 8th Mile, Mysore Road, Bangalore, 560059,
India}
\email{jay@isibang.ac.in, jaydeb@gmail.com}

\subjclass[2010]{47B35, 47B20, 30J05, 30H10, 15B05, 46L99}

\keywords{Toeplitz and Laurent operators, Hardy space over polydisc, partial isometries, hyponormal operators, inner functions, power partial isometries}

\begin{abstract}
A Toeplitz operator $T_\varphi$, $\varphi \in L^\infty(\mathbb{T}^n)$, is a partial isometry if and only if there exist inner functions $\varphi_1, \varphi_2 \in H^\infty(\mathbb{D}^n)$ such that $\varphi_1$ and $\varphi_2$ depends on different variables and $\varphi = \bar{\varphi}_1 \varphi_2$. In particular, for $n=1$, along with new proof, this recovers a classical theorem of Brown and Douglas.

We also prove that a partially isometric Toeplitz operator is hyponormal if and only if the corresponding symbol is an inner function in $H^\infty(\mathbb{D}^n)$. Moreover, partially isometric Toeplitz operators are always power partial isometry (following Halmos and Wallen), and hence, up to unitary equivalence, a partially isometric Toeplitz operator with symbol in $L^\infty(\mathbb{T}^n)$, $n > 1$, is either a shift, or a co-shift, or a direct sum of truncated shifts. Along the way, we prove that $T_\varphi$ is a shift whenever $\varphi$ is inner in $H^\infty(\mathbb{D}^n)$.
\end{abstract}

\maketitle

\section{Introduction}

Toeplitz operators are one of the most useful and prevalent objects in matrix theory, operator theory, operator algebras, and its related fields. For instance, Toeplitz operators provide some of the most important links between index theory, $C^*$-algebras, function theory, and non-commutative geometry. See the monograph by Higson and Roe \cite{HR} for a thorough presentation of these connections, and consult the paper by Axler \cite{Axler} for a rapid introduction to Toeplitz operators.

Evidently, a lot of work has been done in the development of one variable Toeplitz operators, and it is still a subject of very active research, with an ever-increasing list of connections and applications. But on the other hand, many questions remain to be settled in the several variables case, and more specifically in the open unit polydisc case (however, see \cite{D, DSZ, Gu, MSS, SZ}). The difficulty lies in the obvious fact that the standard (and classical) single variable tools are either unavailable or not well developed in the setting of polydisc. Evidently, advances in Toeplitz operators on the polydisc have frequently resulted in a number of new tools and techniques in operator theory, operator algebras, and related fields.

Our objective of this paper is to address the following basic question: Characterize partially isometric Toeplitz operators on $H^2(\D^n)$, where $H^2(\D^n)$ denotes the Hardy space over the unit polydisc $\D^n$. Recall that a partial isometry \cite{HM} is a bounded linear operator whose restriction to the orthogonal complement of its null space is an
isometry.

Before we answer the above question, we first recall that $H^2(\D^n)$ is the Hilbert space of all analytic functions $f$ on $\D^n$ such that
\[
\|f\| := \Big(\sup_{0 \leq r < 1} \int_{\mathbb{T}^n} |f(rz_1, \ldots, r z_n)|^2 d\bm{m}({z}) \Big)^{\frac{1}{2}} < \infty,
\]
where $d\bm{m}(z)$ is the normalized Lebesgue measure on the $n$-torus $\mathbb{T}^n$, and $z = (z_1, \ldots, z_n)$. We denote by $L^2(\T^n)$ the Hilbert space $L^2(\T^n, d\bm{m}({z}))$. From the radial limits of square summable analytic functions point of view \cite{WR}, one can identify $H^2(\D^n)$ with a closed subspace $H^2(\T^n)$ of $L^2(\T^n)$. Let $L^\infty(\T^n)$ denote the standard $C^*$-algebra of $\mathbb{C}$-valued essentially bounded Lebesgue measurable functions on $\T^n$. The \textit{Toeplitz operator} $T_\vp$ with symbol $\vp \in L^\infty(\T^n)$ is defined by
\[
T_\vp f = P_{H^2(\D^n)} (\vp f) \qquad \quad (f \in H^2(\D^n)),
\]
where $P_{H^2(\D^n)}$ denotes the orthogonal projection from $L^2(\T^n)$ onto $H^2(\D^n)$. Also recall that
\[
H^\infty(\D^n) = L^\infty(\T^n) \cap H^2(\D^n),
\]
where $H^\infty(\D^n)$ denotes the Banach algebra of all bounded analytic functions on $\D^n$. A function $\vp \in H^\infty(\D^n)$ is called \textit{inner} if $\vp$ is unimodular on $\T^n$.

The answer to the above question is contained in the following theorem:
\begin{theorem}\label{thm- main}
Let $\vp$ be a nonzero function in $L^\infty(\T^n)$. Then $T_\vp$ is a partial isometry if and only if there exist inner functions $\vp_1, \vp_2 \in H^\infty(\D^n)$ such that $\vp_1$ and $\vp_2$ depends on different variables and
\[
T_\vp = T_{\vp_1}^* T_{\vp_2}.
\]
\end{theorem}

In particular, if $n=1$, then the only nonzero Toeplitz operators that are partial isometries are those of the form $T_\vp$ and $T_\vp^*$, where $\vp \in H^\infty(\D)$ is an inner function. This was proved by Brown and Douglas in \cite{BD}. Actually, as we will see soon in this case that $T_\vp$ is not only an isometry but a shift.

A key ingredient in the proof of the Brown and Douglas theorem is the classical Beurling theorem \cite{Beurling}. Recall that the Beurling theorem connects inner functions in $H^\infty(\D)$ with shift invariant subspaces of $H^2(\D)$. However, in the present case of higher dimensions, this approach does not work, as is well known, Beurling type classification does not hold for shift invariant subspaces of $H^2(\D^n)$, $n >1$ (however, see the proof of Theorem \ref{thm: phi is shift}). Here, we exploit more analytic and geometric structure of $H^2(\D^n)$ and $L^2(\T^n)$ to achieve the main goal. Section \ref{sec:main} contains the proof of the above theorem.

Along the way to the proof of Theorem \ref{thm- main}, in Section \ref{sec: prep} we prove some basic properties of Toeplitz operators on the polydisc. Some of these observations are perhaps known (if not readily available in the literature) to experts, but they are necessary for our purposes here. We also remark that the proof of $\|T_\vp\| = \|\vp\|_{\infty}$, $\vp \in L^\infty(\T^n)$, in Proposition \ref{prop: norm T phi} seems to be different even in the case of $n=1$, as it avoids the standard techniques of the spectral radius formula (see Brown and Halmos \cite[page 99]{BH} and the monographs \cite{Douglas book, Martinez book, Nikolski}).

Moreover, in Section \ref{sec: inner function}, we prove the following result, which connects inner functions with shift operators, and is also of independent interest: \textit{If $\vp \in H^\infty(\D^n)$ is a nonconstant inner function, then $M_\vp$ is a shift.}

\noindent Here, and in what follows, $M_\vp$ denotes the analytic Toeplitz operator $T_\vp$ whenever $\vp \in H^\infty(\D^n)$. In this case, $M_\vp$ is simply the standard multiplication operator on $H^2(\D^n)$, that is, $M_\vp f = \vp f$ for all $f \in H^2(\D^n)$.

In Section \ref{sec: hypo}, as a first application to Theorem \ref{thm- main}, we classify partially isometric hyponormal Toeplitz operators. Recall that a bounded linear operator $T$ on some Hilbert space is called hyponormal if $T^* T - T T^* \geq 0$. In Corollary \ref{coro-hypo}, we prove the following: If $T_\vp$, $\vp \in L^\infty(\T^n)$, is a partial isometry, then $T_\vp$ is hyponormal if and only if $\vp$ is an inner function in $H^\infty(\D^n)$.

\noindent Secondly, following the Halmos and Wallen \cite{HW} notion of power partial isometries (also see an Huef, Raeburn and Tolich \cite{Raeburn}), in Corollary \ref{coro-power partial} we prove that partially isometric Toeplitz operators are always power partial isometry. In Theorem \ref{thm: shift, co-shift, trunc}, we further exploit the Halmos and Wallen models of power partial isometries, and obtain a connection between partially isometric Toeplitz operators, shifts, co-shifts, and direct sums of truncated shifts.

Finally, collecting all these results together, from an operator theoretic point of view, we obtain the following refinement of Theorem \ref{thm- main}:

\noindent \textit{Suppose $T_\vp$, $\vp \in L^\infty(\T^n)$, is partially isometric. Then, up to unitary equivalence, $T_\vp$ is either a shift, or a co-shift, or a direct sum of truncated shifts.}

\noindent We stress that the latter possibility is only restricted to the $n>1$ case.

\section{Preparatory results}\label{sec: prep}

In this section, we develop the necessary tools leading to the proof of Theorem \ref{thm- main}. In this respect, we again remark that in what follows, we will often identify (via radial limits) $H^2(\D^n)$ with $H^2(\T^n)$ without further explanation. Given $\vp \in L^\infty(\T^n)$, we denote by $L_\vp$ the \textit{Laurent operator} on $L^2(\T^n)$, that is, $L_\vp f = \vp f$ for all $f \in L^2(\T^n)$. Note that
\[
\|L_\vp\|_{\clb(L^2(\T^n))} = \|\vp\|_\infty,
\]
where $\|\vp\|_\infty$ denotes the essential supremum norm of $\vp$. The Toeplitz operator $T_\vp$ with symbol $\vp \in L^\infty(\T^n)$ is given by
\[
T_\vp = P_{H^2(\D^n)} L_\vp|_{H^2(\D^n)}.
\]
Clearly, $T_\vp \in \clb(H^2(\D^n))$. Also note that a function $f = \sum\limits_{k \in \Z^n} a_{k} z^{k} \in L^2(\T^n)$ is in $H^2(\D^n)$ if and only if $a_{k} = 0$ whenever at least one of the $k_j$, $j = 1, \ldots, n$, in $k = (k_1, \ldots, k_n)$ is negative.

We start with a several variables analogue of brother Riesz theorem. We denote the set of zeros of a scalar-valued function $f$ by $\clz(f)$.

\begin{lemma}\label{lemma: f = 0}
If $f \in H^2(\D^n)$ is nonzero, then $\bm{m}(\clz(f)) =0$.
\end{lemma}
\begin{proof}
Let $m$ denote the normalized Lebesgue measure on $\T$. Suppose $f$ is a nonzero function in $H^2(\D^2)$. For $w_1$ and $w_2$ in $\T$ a.e., we define the slice functions $f_{w_1}$ and $f_{w_2}$ by $f_{w_1}(z) = f(w_1, z)$ and $f_{w_2}(z) = (z, w_2)$ for all $z \in \T$. Set
\[
\mathcal{Z} = \{w_2 \in \T: f_{w_2} \equiv 0\}.
\]
Note that $\mathcal{Z} \subseteq \clz(f_{w_1})$ for all $w_1 \in \T$. If $m(\mathcal{Z}) >0$, then the classical brother Riesz theorem implies that $f$ is identically zero. Therefore, $m(\mathcal{Z}) = 0$. Evidently
\[
m(\mathcal{Z}(f_{w_2})) =
\begin{cases}
1 & \mbox{if } w_2 \in \mathcal{Z} \\
0 & \mbox{if } w_2 \in \mathcal{Z}^c,
\end{cases}
\]
and hence $w_2 \mapsto m(\mathcal{Z}(f_{w_2}))$ is a measurable function. By the Tonelli and Fubini theorem, we see that
\[
\begin{split}
(m \times m)(\clz(f)) & = \int_{\T} m(\mathcal{Z}(f_{z_2})) \; dm(z_2)
\\
& = \int_{\mathcal Z} m(\mathcal{Z}(f_{z_2})) \; dm(z_2) + \int_{\mathcal{Z}^c} m(\mathcal{Z}(f_{z_2})) \; dm(z_2)
\\
& = 0.
\end{split}
\]
The rest of the proof now follows easily by the induction on $n$.
\end{proof}

We refer to Rudin \cite[Theorem 3.3.5]{WR} for a different proof of the above lemma (even in the context of functions in the Nevanlinna class). Also, see \cite{MT} for the same for functions in $H^\infty(\D^n)$. However, the present proof is direct and avoids the use of heavy machinery from function theory.

We now prove that $\|T_\vp\|_{\clb(H^2(\D^n))} = \|\vp\|_\infty$. As we have pointed out already in the introductory section above, this may be known to experts. However, even when $n=1$, the present proof seems to be direct as it avoids the standard techniques of the spectral radius formula. For instance, see the classic monograph \cite[Corollary 7.8]{Douglas book} and the recent monograph \cite[Corollary 3.3.2]{Martinez book}.

\begin{proposition}\label{prop: norm T phi}
$\|T_\vp\| = \|\vp\|_\infty$ for all $\vp \in L^\infty(\T^n)$.
\end{proposition}
\begin{proof}
Let $\cll$ denote the set of Laurent polynomials in $n$ variables. We compute
\[
\begin{split}
\|T_\vp\| & = \sup \{|\langle \vp f, g \rangle|: f, g \in H^2(\D^n), \|f\|, \|g\| \leq 1\}
\\
& = \sup \{|\langle \vp f, g \rangle|: f, g \in \mathbb{C}[z_1, \ldots, z_n], \|f\|, \|g\| \leq 1\} \quad \text{ (by density of polynomials)}
\\
& = \sup \{|\langle \vp f, g \rangle|: f, g \in \cll, \|f\|, \|g\| \leq 1\}
\\
& = \|L_\vp\|
\\
& = \|\vp\|_\infty.
\end{split}
\]
Note the third equality follows because any Laurent polynomial can be multiplied by a monomial to put it into polynomials. This completes the proof of the proposition.
\end{proof}

The above elegant proof is due to Professor Greg Knese and replaces our original proof, which was longer and technical.

Before proceeding to the proof of the main theorem, we conclude this section with a result concerning unimodular functions in $L^\infty(\T^n)$.

\begin{corollary}\label{lemma: norm attain}
Suppose $\vp$ is a nonzero function in $L^\infty(\T^n)$. If $\|T_\vp f\| = \|\vp\|_\infty \|f\|$ for some nonzero $f \in H^2(\D^n)$, then $\frac{1}{\|\vp\|_\infty}\vp$ is unimodular in $L^\infty(\T^n)$.
\end{corollary}
\begin{proof}
In view of Proposition \ref{prop: norm T phi}, without loss of generality we may assume that $\|T_\vp\|=1$. Then
\[
\int_{\T^n} |\vp(z)|^2 |f(z)|^2 d\bm{m}(z) = \int_{\T^n} |f(z)|^2 d\bm{m}(z).
\]
By Lemma \ref{lemma: f = 0}, $|\vp(z)| = 1$ for all $z \in \T^n$ a.e. and the result follows.
\end{proof}

In particular, if $T_\vp$, $\vp \in L^\infty(\T^n)$, is a partial isometry, then $\vp$  is unimodular.

\section{Proof of Theorem \ref{thm- main}}\label{sec:main}

In this section, without explicitly mentioning it in each instance, we always assume that $T_\vp$, $\vp \in L^\infty(\T^n)$, is partially isometric. Also, we frequently make use of the identification $H^2(\D^n) \cong H^2(\T^n)$ without mentioning it (see Section \ref{sec: prep}).

For simplicity we denote by $\clr(T)$ the range of a bounded linear operator $T$. Clearly, $\clr(T_\vp)$ is a closed subspace of $H^2(\D^n)$.

\begin{lemma}\label{lemma-invariant}
$\clr(T_\vp)$ is invariant under $M_{z_i}$, $i=1,\ldots, n$.
\end{lemma}
\begin{proof}
Note that, since $\|T_\vp\| = 1$, we have $\|\vp\|_{\infty} = 1$. Suppose $f \in \clr(T_\vp)$. By Corollary \ref{lemma: norm attain}, it follows that $\vp$  is unimodular, and hence $\|L_{\bar{\vp}} f\| = \|f\|$. Since $T^*_\vp$ is an isometry on $\clr(T_\vp)$, we have
\[
\|f\| = \|T_\vp^* f\| \leq \|L_{\bar{\vp}} f\| = \|\bar{\vp}f\| = \|f\|.
\]
Therefore, $\|P_{H^2(\D^n)} ({\bar{\vp}} f) \|  = \|{\bar{\vp}} f\|$, that is, $P_{H^2(\D^n)} ({\bar{\vp}} f) = \bar{\vp} f$. This implies that
\begin{equation}\label{eqn: phi f in H2}
\bar{\vp} f \in H^2(\D^n),
\end{equation}
and hence $z_i \bar{\vp} f \in H^2(\D^n)$ for all $i=1, \ldots, n$. Then
\[
T_\vp T_\vp^* (z_i f) = T_\vp (\bar{\vp} z_i f) = P_{H^2(\D^n)} (|\vp|^2 z_i f) = P_{H^2(\D^n)} (z_i f) = z_i f,
\]
implies that $z_i f \in \clr(T_\vp)$ for all $i = 1, \ldots, n$. This completes the proof.
\end{proof}

In what follows, if $i\in \{1, \ldots, n\}$ and $k_i$ is a negative integer, then we write $z_i^{k_i} = \bar{z}_i^{-k_i}$.

\begin{lemma}\label{prop - main}
For each $i = 1, \ldots, n$, the function $\vp$ cannot depend on both $z_i$ and $\bar{z}_i$ variables at a time.
\end{lemma}
\begin{proof}
We shall prove this by contradiction. Assume without loss of generality that $\vp$ depends on both $z_1$ and $\bar{z}_1$. Then
\[
\vp = \sum_{k=1}^{\infty} \bar{z}_1^k \vp_{-k} \oplus \sum_{k=0}^{\infty} z_1^k \vp_k,
\]
and  $\vp_{-k_0} \neq 0$ for some $k_0 \neq 0$. Here $\vp_k \in L^2(\T^{n-1})$, $k \in \mathbb{Z}$, is a function of $\{z_i, \bar{z}_j: i,j=2, \ldots, n\}$. There exist non-negative integers $k_2,\ldots, k_{n}$, and $l_2, \ldots, l_n$ such that the coefficient of $\bar{z}_2^{k_2} \cdots \bar{z}_n^{k_n} z_2^{l_2}\cdots z_n^{l_n}$ in the expansion of the Fourier series of $\vp_{-k_0}$ is nonzero. Set
\[
Z_{kl}:={z}_2^{k_2} \cdots {z}_n^{k_n} z_2^{l_2}\cdots z_n^{l_n},
\]
and
\[
f:= T_\vp (z_1^{k_0} Z_{kl}) - z_1 T_\vp(z_1^{k_0-1}Z_{kl}).
\]
Note that $f$ is a nonzero function in $H^2(\D^n)$, and $f$ does not depend on $z_1$. Since $T_\vp(z_1^{k_0-1}Z_{kl}) \in \clr(T_\vp)$, Lemma \ref{lemma-invariant} implies that $f \in \clr(T_\vp)$. In particular, by \eqref{eqn: phi f in H2}, $\bar{\vp} f \in H^2(\D^n)$. On the other hand, since
\[
\bar{\vp} f = \sum_{k=1}^{\infty} {z}_1^k (f\bar{\vp}_{-k}) \oplus \sum_{k=0}^{\infty} \bar{z}_1^k (f\bar{\vp}_k),
\]
it follows that $f \bar{\vp}_k = 0$ for all $k > 0$. Since $\bm{m}(\{z \in \mathbb{T}^n: f(z) = 0\}) =0$, we have $\bar{\vp}_k = 0$ for all $k > 0$. This yields
\[
\vp = \sum_{k=0}^{\infty} \bar{z}_1^k \vp_{-k},
\]
and hence $\vp$ depends on $\bar{z}_1$ and does not depend on $z_1$. This is a contradiction.
\end{proof}

We are now ready for the proof of Theorem \ref{thm- main}.

\smallskip

\begin{proof}[Proof of Theorem \ref{thm- main}]
Suppose $T_\vp$ is a partial isometry. In view of Lemma \ref{prop - main}, there exists a (possibly empty) subset $C$ of $\{z_1, \ldots, z_n\}$ such that $\vp$ is analytic in $z_i$ for all $z_i \in A :=C^c$, and co-analytic in $z_j$ for all $z_j \in C$. Let $A = \{z_{i_1}, \ldots, z_{i_p}\}$ and $C = \{z_{j_1}, \ldots, z_{j_q}\}$. Then $p+q=n$, and
\[
\vp = \sum_{k \in \Z_+^{q}} \bar{z}_C^k \vp_{A,k},
\]
where $\vp_{A,k} \in H^2(\D^p)$ is a function of $\{z_{i_1}, \ldots, z_{i_p}\}$, $\bar{z}_C^k = \bar{z}_{j_1}^{k_1} \cdots \bar{z}_{j_q}^{k_q}$, and $k = (k_1, \ldots, k_q)\in \Z_+^{q}$. Note that
\[
\vp_{A,l} \in \clr(T_\vp) \qquad (l \in \Z_+^{q}).
\]
Indeed, $\vp_{A,0} = T_\vp 1 \in \clr(T_\vp)$. Moreover, for each $l \in \Z_+^{q} \setminus \{0\}$, we have
\[
T_\vp z^l = P_{H^2(\D^n)} \Big(\sum_{k \in \Z_+^{q}} {z}_C^{l-k} \vp_{A,k}\Big),
\]
that is
\[
T_\vp z^l = \sum_{l - k \geq 0} {z}_C^{l-k} \vp_{A,k}.
\]
Here $l-k \geq 0$ means that $l_i - k_i \geq 0$ for all $i=1, \ldots, q$. Thus the claim follows by induction. By  \eqref{eqn: phi f in H2}, we have $\bar{\vp} \vp_{A,l} \in H^2(\D^n)$, $l \in \Z_+^{q}$. Therefore
\[
\bar{\vp} \vp_{A,l} = \sum_{k \in \Z_+^{q}} {z}_C^k \overline{\vp_{A,k}}\vp_{A,l} \in H^2(\D^n) \qquad (l \in \Z_+^{q}).
\]
Consequently, $\overline{\vp}_{A,k} \vp_{A,l} \in H^2(\D^p)$ for all $k$ and $l$, and hence, in particular, we have
\[
\overline{\vp}_{A,l}\vp_{A,l} \in H^2(\D^p) \qquad (l \in \Z^q_+).
\]
This immediately implies that $\overline{\vp}_{A,l}\vp_{A,l}$ is a constant function, and hence $\vp_{A,l} = \alpha_l {\psi_l}$ for some inner function $\psi_l \in H^\infty(\D^p)$ and scalar $\alpha_l$ such that $|\alpha_l| \leq 1$, $l \in \Z^q_+$. Assume without loss of generality that $\vp_{A,0} \neq 0$. Now by the fact that $\overline{\vp}_{A,0}\vp_{A,k}$ and $\overline{\vp}_{A,k}\vp_{A,0}$ are in $H^2(\D^p)$, we have $\vp_{A,k} = \beta_k \psi_0$, $k \in \Z^q_+$. Therefore
\[
\vp = \Big(\sum_{k \in \Z_+^{q}} \beta_k \bar{z}_C^k\Big) \psi_0 = \bar{\vp}_1 \vp_2,
\]
where $\vp_1 = \sum_{k \in \Z_+^{q}} \bar{\beta}_k {z}_C^k$ and $\vp_2 = \psi_0$.

We now turn to the converse part. First we have clearly
\begin{equation}\label{eqn: commutes}
T_{\vp_1} T_{\vp_2} = T_{\vp_2} T_{\vp_1}.
\end{equation}
We also claim that
\begin{equation}\label{eqn: 12 equal}
T_{\vp_1} T_{\vp_2}^* = T_{\vp_2}^* T_{\vp_1}.
\end{equation}
This holds trivially when one of the functions $\vp_1$ or $\vp_2$ is constant. We continue with the above notation, and assume that both $A$ and $C$ are nonempty subsets of $\{z_1, \ldots, z_n\}$. First we observe that $\vp_1$ and $\vp_2$ depends only on $\{z_{i_1}, \ldots, z_{i_p}\}$ and $\{z_{j_1}, \ldots, z_{j_q}\}$, respectively. Consider a monomial $z^k \in \mathbb{C}[z_1, \ldots, z_n]$. Suppose $k = (k_1, \ldots, k_n)$, and write
\[
z^k = z_C^{k_c} z_A^{k_a},
\]
where $k_c = (k_{j_1}, \ldots, k_{j_q}) \in \Z_+^q$, and $k_a \in \Z_+^{p}$ is the ordered $p$ tuple made out of $\{k_i\}_{i=1}^n \setminus \{k_{j_t}\}_{t=1}^q$. Since the analytic function $\vp_2$ depends only on $z_{j_s} \in C$, $s=1, \ldots, p$, it is clear that
\[
\bar{\vp}_2 z_C^{k_c} = \vp_a + \vp_c,
\]
where $\vp_a$ depends only on $\{z_{j_s}\}_{s=1}^p$ (and hence it is an analytic function) and $\vp_c \in L^2(\T^q) \ominus H^2(\D^q)$ is a function of $\{{z}_{j_t}, \bar{z}_{j_t}\}_{t=1}^q$. Note that the latter property ensures that $\vp_c(0) = 0$. Then, on one hand, we have
\[
T_{\vp_2}^* T_{\vp_1} z^k = P_{H^2(\D^n)} (\bar{\vp}_2 \vp_1 z^k) = P_{H^2(\D^n)} \Big( (\vp_a + \vp_c) \vp_1 z_A^{k_a}\Big) = \vp_a \vp_1 z_A^{k_a},
\]
and on the other hand that
\[
T_{\vp_1} T_{\vp_2}^* z^k = \vp_1 P_{H^2(\D^n)} (\bar{\vp}_2 z^k) = \vp_1 P_{H^2(\D^n)} \Big( (\vp_a + \vp_c) z_A^{k_a}\Big) = \vp_1 \vp_a z_A^{k_a}.
\]
Consequently, $T_{\vp_2}^* T_{\vp_1} z^k = T_{\vp_1} T_{\vp_2}^* z^k$ for all $k \in \Z_+^n$, which proves our claim. Now suppose that $T_\vp = T_{\vp_1}^* T_{\vp_2}$, where $\vp_1$ and $\vp_2$ depends on different variables. Using \eqref{eqn: commutes} and \eqref{eqn: 12 equal}, we obtain
\begin{equation}\label{eqn: vp vp*}
T_\vp T_\vp^* = T_{\vp_1}^* T_{\vp_2} T_{\vp_2}^* T_{\vp_1} = (T_{\vp_1}^* T_{\vp_1}) (T_{\vp_2}T_{\vp_2}^*) = P_{\clr(T_{\vp_2})},
\end{equation}
which implies that $T_\vp$ is a partial isometry.
\end{proof}

We remark that the commutativity and doubly commutativity of $T_{\vp_1}$ and $T_{\vp_2}$ in \eqref{eqn: commutes} and \eqref{eqn: 12 equal} will be useful in the particular applications to Theorem \ref{thm- main} in the final section.

\section{Inner functions and shifts}\label{sec: inner function}

In this short section, we pause to prove an auxiliary result that is both a necessary tool for our final refinement of partial isometric Toeplitz operators and a subject of independent interest with its own applications.

Let $\vp \in H^\infty(\D^n)$, and suppose the multiplication operator $M_\vp$ is an isometry on $H^2(\D^n)$. Then
\[
\|\vp\|_{\infty} = \|M_\vp\|_{\clb(H^2(\D^n))} = 1,
\]
and hence Corollary \ref{lemma: norm attain} implies that $\vp$ is a unimodular function in $H^\infty(\D^n)$, that is, $\vp$ is an inner function. Now we prove that a nonconstant inner function always defines a shift (and not only isometry). Recall that an operator $V \in \clb(\clh)$ is said to be a \textit{shift} if $V$ is an isometry and $V^{*m} \raro 0$ as $m \raro \infty$ in the strong operator topology.

Recall that a closed subspace $\cls \subseteq H^2(\D^n)$ is of \textit{Beurling type} if there exists an inner function $\theta \in H^\infty(\D^n)$ such that $\cls = \theta H^2(\D^n)$. It is also known that (cf. \cite[Corollary 6.3]{MMSS} and \cite{Lu}) a closed subspace $\cls \subseteq H^2(\D^n)$, $n > 1$, is of Beurling type if and only if $R_i^* R_j = R_j R^*_i$ for all $1 \leq i < j \leq n$, where $R_p = M_{z_p}|_{\cls} \in \clb(\cls)$ is the restriction operator and $p=1, \ldots, n$. Note that
\begin{equation}\label{eqn: dc equality}
R_i^* R_j = P_{\cls} M_{z_i}^* M_{z_j}|_{\cls} \mbox{~~and~~} R_j R_i^* = M_{z_j} P_{\cls} M_{z_i}^*|_{\cls},
\end{equation}
for all $i,j=1,\ldots, n$.

\begin{theorem}\label{thm: phi is shift}
If $\vp \in H^\infty(\D^n)$ is a nonconstant inner function, then $M_\vp$ is a shift.
\end{theorem}
\begin{proof}
It is well known (as well as easy to see) that $M_\vp$ is an isometry. Following the classical von Neumann and Wold decomposition for isometries, we only need to prove that
\[
\clh_u:= \bigcap_{m=0}^\infty \vp^m H^2(\D^n) = \{0\}.
\]
Assuming the contrary, suppose that $\clh_u \neq \{0\}$. We claim that $\clh_u$ is of Beurling type. Since the $n=1$ case is obvious, we assume that $n > 1$. As $\vp^p H^2(\D^n) \subseteq \vp^q H^2(\D^n)$ for all $p \geq q$, we have
\[
P_{\clh_u} = SOT-\lim_{m\raro \infty} P_{\vp^m H^2(\D^n)}.
\]
Since $\vp^m H^2(\D^n)$, $m \geq 1$, is a Beurling type invariant subspace, in view of \eqref{eqn: dc equality}, it follows that
\[
P_{\clh_u} M_{z_i}^* M_{z_j} h = M_{z_j} P_{\clh_u} M_{z_i}^*h,
\]
for all $h \in \clh_u$. Then \eqref{eqn: dc equality} again implies that $\clh_u$ is of Beurling type. Therefore, there exists an inner function $\theta \in H^\infty(\D^n)$ such that $\clh_u = \theta H^2(\D^n)$ (note that the $n=1$ case directly follows from Beurling). Then, for each $m \geq 1$, there exists an inner function $\psi_m \in H^\infty(\D^n)$ such that $\theta = \vp^m \psi_m$ (for instance, see \eqref{eqn:phi X commute}). Since $\vp$ is a nonconstant inner function, by the maximum modulus principle \cite[\S 2, Theorem 6]{Shabat}, we have $|\vp(z)| < 1$ for all $z \in \D^n$. For each fixed $z_0 \in \D^n$, it follows that
\[
|\theta(z_0)| = |\vp(z_0)|^m |\psi_m(z_0)| \leq |\vp(z_0)|^m \raro 0 \mbox{~as~} m\raro \infty,
\]
and hence $\theta \equiv 0$. This contradiction shows that $\clh_u = \{0\}$.
\end{proof}

In fact, the above argument yields something more: Suppose $\{\cls_m\}_{m\geq 1}$ be a sequence of Beurling type invariant subspaces of $H^2(\D^n)$. Then $\bigcap_{m=1}^\infty \cls_m$ is also a Beurling type invariant subspace. Indeed, we let $\clh_m = \bigcap_{i=1}^m \cls_m$. Then $\{\clh_m\}_{m\geq 1}$ forms a decreasing sequence of Beurling type invariant subspaces, and hence
\[
P_{\bigcap_{m=1}^\infty \cls_m} = P_{\bigcap_{m=1}^\infty \clh_m} = SOT-\lim_{m\raro \infty} P_{\clh_m}.
\]
The rest of the proof is then much as before.

We also wish to point out that Theorem \ref{thm: phi is shift} can be proved by using (analytic) reproducing kernel Hilbert space techniques. We believe that the algebraic tools described above might be useful in other settings.

\section{Applications and further refinements}\label{sec: hypo}

We begin with partially isometric Toeplitz operators that are hyponormal. A bounded linear operator $T$ acting on a Hilbert space is called \textit{hyponormal} if $[T^*, T] \geq 0$, where
\[
[T^*, T] = T^* T - T T^*,
\]
is the self commutator of $T$.

Now suppose $T_\vp$, $\vp \in L^\infty(\T^n)$, is a partial isometry. If $\vp \in H^\infty(\D^n)$ is inner, then $T_\vp$ is an isometry and hence is hyponormal. For the converse direction, we note by Theorem \ref{thm- main} that $T_\vp = T_{\vp_1}^* T_{\vp_2}$ for some inner functions $\vp_1$ and $\vp_2$ in $H^\infty(\D^n)$ which depends on different variables. If $\vp_1$ is a constant function, then $T_\vp = T_{\vp_2} = M_{\vp_2}$ is an isometry, and hence $T_\vp$ is hyponormal. If $\vp_2$ is a constant function, then $T_\vp = T_{\vp_1}^* = M_{\vp_1}^*$ is a co-isometry, and hence $T_\vp$ cannot be hyponormal. Suppose both $\vp_1$ and $\vp_2$ are nonconstant functions. Now \eqref{eqn: commutes} and \eqref{eqn: 12 equal} imply that
\[
T_\vp^* T_\vp = T_{\vp_2}^* T_{\vp_1} T_{\vp_1}^* T_{\vp_2} = (T_{\vp_2}^* T_{\vp_2}) (T_{\vp_1}T_{\vp_1}^*) = T_{\vp_1}T_{\vp_1}^*.
\]
Then, by \eqref{eqn: vp vp*} we see that $[T_{\vp}^*, T_\vp] \geq 0$ implies $T_{\vp_2} T_{\vp_2}^* \leq T_{\vp_1}T_{\vp_1}^*$. By noting that $\vp_1$ and $\vp_2$ are analytic functions, we see
\[
M_{\vp_2} M_{\vp_2}^* \leq M_{\vp_1} M_{\vp_1}^*,
\]
which, by the Douglas range inclusion theorem, is equivalent to $M_{\vp_2} = M_{\vp_1} X$ for some $X \in \clb(H^2(\D^n))$. Observe that
\begin{equation}\label{eqn:phi X commute}
M_{\vp_1} M_{z_i} X = M_{z_i} M_{\vp_1} X = M_{z_i} M_{\vp_2} = M_{\vp_2} M_{z_i} = M_{\vp_1} X M_{z_i},
\end{equation}
implies that $M_{z_i} X = X M_{z_i}$ for all $i=1, \ldots, n$, and hence $X = M_{\psi}$ for some $\psi \in H^\infty(\D^n)$. Hence, we conclude that $\vp_2 = \vp_1 \psi$. Since $\vp_1$ and $\vp_2$ are inner functions, $\psi \in H^\infty(\D^n)$ is inner. Moreover, since $\vp_1$ and $\vp_2$ depends on different variables, that $\vp_2 = \vp_1 \psi$ is possible if and only if $\psi$ is a unimodular constant. Suppose $\vp_2 = \alpha \vp_1$, where $|\alpha| = 1$. Then
\[
T_\vp = T_{\vp_1}^* T_{\vp_2} = \alpha T_{\vp_1}^* T_{\vp_1} = \alpha T_{|\vp_1|^2} = \alpha I_{H^2(\D^n)},
\]
as $|\vp_1|^2 = 1$ on $\T^n$, that is, $T_\vp$ is a unimodular constant times the identity operator. We have therefore shown the following result:

\begin{corollary}\label{coro-hypo}
Let $T_\vp$, $\vp \in L^\infty(\T^n)$, be a partial isometry. Then $T_\vp$ is hyponormal if and only if $\vp$ is an inner function in $H^\infty(\D^n)$.
\end{corollary}

Therefore, in view of Theorem \ref{thm: phi is shift}, $T_\vp$ is hyponormal if and only if (up to unitary equivalence) $T_\vp$ is a shift.

We recall \cite[Halmos and Wallen]{HW} that a bounded linear operator $T$ acting on some Hilbert space is called a \textit{power partial isometry} if $T^m$ is partially isometric for all $m \geq 1$. Clearly, Theorem \ref{thm- main} and the equalities in \eqref{eqn: commutes} and \eqref{eqn: 12 equal} imply the following statement:

\begin{corollary}\label{coro-power partial}
Partially isometric Toeplitz operators are power partial isometry.
\end{corollary}

We also recall from Halmos and Wallen \cite{HW} (also see \cite{Raeburn}) that every power partial isometry is a direct sum whose summands are unitary operators, shifts, co-shifts, and truncated shifts. Recall that a \textit{truncated shift} $S$ of \textit{index} $p$, $p \in \mathbb{N}$, on some Hilbert space $\clh$ is an operator of the form
\[
S= \begin{bmatrix}
0 & 0 & 0 & \cdots  & 0& 0
\\
I_{\clh_0} & 0 & 0 & \cdots & 0& 0
\\
0 & I_{\clh_0} & 0 & \cdots & 0& 0
\\
\vdots & \vdots & \vdots & \cdots & 0 & 0
\\
0 & 0 & 0 & \cdots & I_{\clh_0} & 0
\end{bmatrix}_{p \times p},
\]
where $\clh_0$ is a Hilbert space, and $\clh = \underbrace{\clh_0 \oplus \cdots \oplus \clh_0}_{p}$.

We prove that, up to unitary equivalence, a partial isometric $T_\vp$ is simply direct sum of truncated shifts, or a shift, or a co-shift (that is, adjoint of a shift). The proof is essentially contained in Theorem \ref{thm: phi is shift} and the Halmos and Wallen models of power partial isometries.

\begin{theorem}\label{thm: shift, co-shift, trunc}
Up to unitary equivalence, a partially isometric Toeplitz operator is either a shift, or a co-shift, or a direct sum of truncated shifts.
\end{theorem}
\begin{proof}
Suppose $T_\vp$, $\vp \in L^\infty(\T^n)$, is a partial isometry. By Theorem \ref{thm- main}, $T_\vp = T_{\vp_1}^* T_{\vp_2}$, where $\vp_1$ and $\vp_2$ are inner functions in $H^\infty(\D^n)$ and depends on different variables. Moreover, by Corollary \ref{coro-power partial}, $T_\vp$ is a power partial isometry. If $\vp_1$ is a constant function, then $T_\vp$ is a shift, and if $\vp_2$ is a constant function, then $T_\vp$ is a co-shift. Now let both $\vp_1$ and $\vp_2$ are nonconstant functions. Following the construction of Halmos and Wallen \cite[page 660]{HW} (also see \cite{Raeburn}), we set $E_m = T_{\vp}^{*m} T_{\vp}^m$ and $F_m = T_{\vp}^m T_{\vp}^{*m}$ for the initial and final projections of the partial isometry $T_{\vp}^m$, $m \geq 1$. By \eqref{eqn: commutes} and \eqref{eqn: 12 equal} it follows that $E_m = T_{\vp_1}^{m} T_{\vp_1}^{*m}$ and $F_m = T_{\vp_2}^m T_{\vp_2}^{*m}$, and hence
\[
\clr(E_m) = \vp_1^m H^2(\D^n) \mbox{~~and~~} \clr(F_m) = \vp_2^m H^2(\D^n),
\]
for all $m \geq 1$. Then, by Theorem \ref{thm: phi is shift}, we have
\[
\bigcap_{m\geq 0} \clr(E_m) = \bigcap_{m\geq 0} \vp_1^m H^2(\D^n) = \{0\},
\]
and similarly $\bigcap\limits_{m\geq 0} \clr(F_m) = \{0\}$. Therefore, the unitary part, the shift part, and the co-shift part of the Halmos and Wallen model of $T_\vp$ are trivial (see \cite[page 661]{HW} or \cite{Raeburn}). Hence in this case, $T_\vp$ is a direct sum of truncated shifts.
\end{proof}

Clearly, Corollary \ref{coro-hypo} immediately follows from the above result as well. Also, note that the Halmos and Wallen models of power partial isometries played an important role in the proof of the above theorem. We refer \cite{Raeburn, Bracci, FR} for a more recent view point of power partial isometries.

Finally, summarizing our results from an operator theoretic point of view, we conclude the following: Let $T_\vp$, $\vp \in L^\infty(\T^n)$, be a partially isometric Toeplitz operator. Then the following hold:

\begin{enumerate}
\item If $n=1$, then $T_\vp$ is either an isometry, or a coisometry. This is due to Brown and Douglas. And, in view of Theorem \ref{thm: phi is shift}, $T_\vp$ is either a shift, or a co-shift.
\item If $n > 1$, then, up to unitary equivalence, $T_\vp$ is either a shift, or a co-shift, or a direct sum of truncated shifts.
\end{enumerate}

\vspace{0.2in}

\noindent\textbf{Acknowledgement:}
We are very thankful to Professor Greg Knese for his comments and suggestions. In particular, the proof of Proposition \ref{prop: norm T phi} is due to Professor Knese, and we are grateful to him for allowing us to include his argument here. The research of the second named author is supported by NBHM  (National Board of Higher Mathematics, India) post-doctoral fellowship no: 0204/27-/2019/R\&D-II/12966. The third author is supported in part by the Mathematical Research Impact Centric Support, MATRICS (MTR/2017/000522), and Core Research Grant (CRG/2019/000908), by SERB, Department of Science \& Technology (DST), and NBHM (NBHM/R.P.64/2014), Government of India.

\end{document}